\documentclass[12pt]{amsart}
\usepackage{latexsym,amsmath, amscd, amssymb, amsthm, euscript, mathrsfs, stmaryrd}
\usepackage[all]{xy}
\usepackage{hyperref}
\usepackage[colorinlistoftodos]{todonotes}
\usepackage[cal=boondoxo, scr=boondoxo]{mathalfa}
\usepackage{graphicx}
\graphicspath{ {./images/} }

\newtheorem{thm}[subsection]{Theorem}

\newtheorem{thm/def}[subsection]{Theorem/Definition}

\newtheorem{cor}[subsection]{Corollary}

\newtheorem{lem}[subsection]{Lemma}

\newtheorem{prop}[subsection]{Proposition}

\theoremstyle{definition}

\newtheorem{defn}[subsection]{Definition}

\theoremstyle{definition}

\theoremstyle{definition}

\newtheorem{rem}[subsection]{Remark}

\newtheorem{example}[subsection]{Example}
\numberwithin{equation}{subsection}

\newtheorem{pg}[subsection]{}

\newcommand{\mc}{\mathcal }

\newcommand{\dlog}{\text{\rm dlog}}

\newcommand{\Sec}{{\text{\rm Sec}}}

\newcommand{\mls}{\mathscr}

\newcommand{\cPic}{\mathscr{P}\!{\it ic}}

\DeclareMathOperator{\Pic}{\text{\rm Pic}}

\newcommand{\et}{\text{\rm \'et}}

\usepackage{tikz-cd}

\newcommand{\cX}{\mathcal{X}}
\newcommand{\cY}{\mathcal{Y}}
\newcommand{\cR}{\mathcal{R}}
\newcommand{\cAut}{\mathcal{A}ut}
\newcommand{\Aut}{\text{\rm Aut}}

\newcommand{\cSec}{\mathcal{S}ec}
\newcommand{\lotimes}{\otimes ^{\mathbf{L}}}
\newcommand{\rR}{\mathrm{R}}
\newcommand{\rL}{\mathrm{L}}
\newcommand{\BR}{\text{\rm BR}}

\usepackage{rotating}

\AtBeginDocument{%
   \def\MR#1{}
}

\begin{document}

\title{Twisted derived categories and Rouquier functors}

 \author{Martin Olsson}
 
 \begin{abstract}
 We study the algebraic structure of the automorphism group of the derived category of coherent sheaves on a  smooth projective variety twisted by a Brauer class.  Our main results generalize results of Rouquier in the untwisted case. 
 \end{abstract}

\maketitle

\section{Statements of results}

For a noetherian algebraic stack $\cX $ we write $D(\cX )$ for the bounded derived category of coherent sheaves on $\cX $.

Let $X$ and $Y$ be smooth projective varieties over an algebraically closed  field $k$ related by a derived equivalence $\Phi :D(X)\rightarrow D(Y)$; that is, an equivalence between their bounded derived categories of coherent sheaves. A fundamental result of Rouquier \cite{Rouquier} is that the equivalence induces (in a manner discussed below) an isomorphism of group schemes 
\begin{equation}\label{E:Rouquieriso}
\Pic ^0_X\times \Aut ^0_X\simeq \Pic ^0_Y\times \Aut ^0_Y,
\end{equation}
where the superscripts ``0'' refer to the connected components of the identity.  The purpose of this article is to explain how this result generalizes to the case of twisted derived categories.

Let  $X$ be a smooth geometrically connected projective variety over $k$ and let $\cX \rightarrow X$ be a $\mathbf{G}_m$-gerbe with associated class $\alpha \in H^2(X, \mathbf{G}_m)$. 
A key role in our interpretation of the Rouquier isomorphism in this context is played by the automorphism group $\Aut _{\cX }$ of the stack $\cX $.    Let $\cAut _{\cX }$ be the fibered category which to any $k$-scheme $S$ associates the groupoid of isomorphisms $\cX \rightarrow \cX$ inducing the identity on the stabilizer group schemes $\mathbf{G}_m$.

\begin{thm}\label{T:theorem1} (1) The fibered category $\cAut _{\cX }$ is an algebraic stack locally of finite type over $k$ which is a $\mathbf{G}_m$-gerbe over a group algebraic space $\Aut _{\cX }$.

(2) If $\Aut ^0_{\cX }\subset \Aut _{\cX }$ denotes the connected component of the identity then there is an exact sequence of group algebraic spaces
$$
1\rightarrow \Pic ^0_X\rightarrow \Aut ^0_{\cX }\rightarrow \Aut ^0_X.
$$

(3)
 If $\cX $ is the pushout of a $\mu _N$-gerbe for $N>0$ invertible in $k$ (this always holds if $k$ has characteristic $0$) then the map $\Aut ^0_{\cX }\rightarrow \Aut ^0_X$ is surjective.
\end{thm}

\begin{example} In the case when $\cX $ is the trivial gerbe $B\mathbf{G}_{m, X}$ the group $\Gamma _{\cX }$ is trivial and the sequence is split. Indeed if $(a, b)\colon X\times B\mathbf{G}_m\rightarrow X\times B\mathbf{G}_m$ is an automorphism of $\mathbf{G}_m$-gerbes then $a$ necessarily factors through an automorphism $\bar a\colon X\rightarrow X$ and $b$ is given by a line bundle $\mc L$ on $X\times B\mathbf{G}_m$ on which the stabilizer groups act through the standard character. From this it follows that $b$ is specified by a line bundle $\mc M$ on $X$ by the formula $\mc L = \mc M\boxtimes \chi $.  It follows that  \begin{equation}\label{E:groupstackiso}
\cAut _{B\mathbf{G}_{m, X}}\simeq \Aut _X\times \mls Pic _X
\end{equation}
as a stack, where $\mls Pic _X$ is the Picard stack of line bundles on $X$. In particular, we have $\Aut _{B\mathbf{G}_{m, X}}\simeq \text{Pic}_X\times \Aut _X$.  The map $\Aut _X\times \mls Pic _X\rightarrow \cAut _{B\mathbf{G}_{m, X}}$ can also be described as follows.  Let $T$ be a $k$-scheme and let $T\rightarrow X\times B\mathbf{G}_m$ be a map corresponding to a pair $(x, \mls L)$ consisting of a $T$-point $x$ of $X$ and a line bundle $\mls L$ on $T$.  Then the automorphism of $B\mathbf{G}_{m, X_T}$ induced by a pair $(\alpha , \mls M)\in \Aut _X(T)\times \mls Pic _X(T)$ sends $(x, \mls L)$ to $(\alpha (x), \mls L\otimes x^*\mls M)$.

Note, however, that \eqref{E:groupstackiso} is not an isomorphism of group stacks.  Given $(\alpha , \mls L), (\alpha ', \mls L')\in \Aut _X\times \mls Pic _X$ the composition of automorphims $(\alpha , \mls M)\circ (\alpha ', \mls M')$ is equal to $(\alpha \circ \alpha ', \mls M'\otimes\alpha ^{\prime *}\mls M)$.  So as a group stack this should be viewed as a semi-direct product $\Aut _X\ltimes \mls Pic_X$.  In particular, we have $\Aut _{\mathbf{G}_{m, X}}\simeq \Aut _X\ltimes \Pic _X$.  If $\Pic ^0_X$ is an abelian variety, then $\Aut ^0_X$ acts trivially on $\Pic ^0_X$ and it follows that the connected component of the identity $\Aut ^0_{B\mathbf{G}_{m, X}}$ is isomorphic as a group  to $\Aut ^0_X\times \Pic ^0_X$.
\end{example}

Our generalization of Rouquier's theorem is phrased in terms of derived invariance of the group $\Aut ^0_{\cX }$.   To phrase our main result in this regard, we first introduce another group stack $\cR_{\cX }$.

Let $A$ be an abelian group and let $D(A):= \text{Hom}(A, \mathbf{G}_m)$ denote the associated diagonalizable group scheme.  The main examples for us are $A = \mathbf{Z}, \mathbf{Z}/N$, $\mathbf{Z}^r$, in which case $D(A) = \mathbf{G}_m, \mu _N, \mathbf{G}_m^r$.  If $\cX \rightarrow X$ is a $D(A)$-gerbe then every object $F\in D(\cX )$ has a canonical decomposition $F = \oplus _{a\in A}F_a$ and morphisms in $D(\cX )$ respects this decomposition.  This is discussed in \cite[\S 2.1]{MR2309155}.  We let $D(\cX )^{(a)}\subset D(\cX )$ denote the subcategory of objects for which $F = F_a$.  An object $K\in D(\cX )$ lies in $D(\cX )^{(a)}$ if and only if for every geometric point $\bar x\rightarrow \cX $ the action of $D(A)$ on the cohomology groups of the fiber $K(\bar x)$ is through the character $a$.

Following existing literature, if $A = \mathbf{Z}$ so that $\cX $ is a $\mathbf{G}_m$-gerbe with associated Brauer class $\alpha \in H^2(X, \mathbf{G}_m)$ then we write $D(X, \alpha )$ for the triangulated category $D(\cX )^{(1)}$.

\begin{defn} Let $\cX /X$ and $\cY/Y$ be two $\mathbf{G}_m$-gerbes over smooth projective varieties over a field $k$ with associated Brauer classes $\alpha \in H^2(X, \mathbf{G}_m)$ and $\beta \in H^2(Y, \mathbf{G}_m)$.  A \emph{Fourier-Mukai functor} $D(X, \alpha )\rightarrow D(Y, \beta )$ is an object $K\in D(\cX \times \cY)^{(-1,1)}$ such that the induced functor
$$
\Phi ^K\colon D(X, \alpha )\rightarrow D(Y, \beta ), \ \ F\mapsto \rR q_*(\rL p^*F\lotimes K)
$$
is an equivalence, where $p\colon \cX \times \cY \rightarrow \cX $ and $q\colon \cX \times \cY \rightarrow \cY$ are the projections.
\end{defn}

\begin{rem}
    Note that $\cX \times \cY$ is a $\mathbf{G}_m^2$-gerbe over $X\times Y$.  If $F\in D(X, \alpha )$ then $\rL p^*F\lotimes K\in D(\cX \times \cY)^{(0,1)}$, and therefore descends (via the derived pushforward functor) to an object of $D(X\times \cY)^{(1)}$.  Since $X/k$ is proper it follows that $\rR q_*(\rL p^*F\lotimes K)$ is a $1$-twisted bounded complex on $\cY$ with coherent cohomology sheaves.  In particular, the functor $\Phi ^K$ is well-defined.
\end{rem}

\begin{example} Let $\sigma \colon \cX \rightarrow \cX $ be an automorphism of a $\mathbf{G}_m$-gerbe, and let $\Gamma _\sigma := (\sigma , \text{id})\colon \cX \rightarrow \cX \times \cX $ be its graph (note that this is not an immersion; in fact, has some fibers of positive dimension).  We can then consider the $(-1,1)$-twisted part $(\Gamma _{\sigma *}\mls O_{\cX })_{(-1,1)}\in D(\cX \times \cY)^{(-1,1)}$.  For $F\in D(X, \alpha )$ we have
$$
\sigma  ^*(F) = \rR q_*(\Gamma _{\sigma *}\mls O_{\cX }\lotimes \rL p^*(F)) = \rR q_*((\Gamma _{\sigma *}\mls O_{\cX })^{(-1,1)}\lotimes \rL p^*F)
$$
and therefore $\Phi ^{(\Gamma _{\sigma *}\mls O_{\cX })_{(-1,1)}} = \sigma ^*$.
\end{example}

\begin{pg}\label{P:Rdef}
Let $\cR _{\cX }$ denote the fibered category which to any $k$-scheme $S$ associates the groupoid of perfect complexes $P\in D((\cX \times \cX )_S)$ such that for all geometric points $\bar s\rightarrow S$ the fiber $P_{\bar s}\in D((\cX \times \cX)_{\kappa (\bar s)})$ is of the form $(\Gamma _{\sigma *}\mls O_{\cX _{\kappa (\bar s)}})_{(-1,1)}$ for an automorphism $\sigma \colon \cX _{\kappa (\bar s)}\rightarrow \cX _{\kappa (\bar s)}$.
\end{pg}

\begin{thm}\label{T:theorem2} (1) The fibered category $\cR _{\cX }$ is an algebraic stack locally of finite type over $k$, which is a $\mathbf{G}_m$-gerbe over a group algebraic space $\mathbf{R}_{\cX }$.

(2) The natural map 
$$
\cAut _{\cX }\rightarrow \cR _{\cX }, \ \ \sigma \mapsto (\Gamma _{\sigma *}\mls O_{\cX })_{(-1,1)}
$$
is an isomorphism of stacks.
\end{thm}

\begin{rem} In general it is not clear that the action of $\cR ^0_{\cX }$ on $D(X, \alpha )$ is faithful.  However, there is a finite flat subgroup scheme $\Gamma \subset \textrm{Pic}^0_X$ such that the action of $\mathbf{R}^0_{\cX }/\Gamma $ on $D(X, \alpha )$ is faithful in an appropriate sense (see \ref{P:faithful}).
\end{rem}

Finally we establish the derived invariance of the group $\cR ^0_{\cX }$.
Let $\mc Y/Y$ be another $\mathbf{G}_m$-gerbe over a smooth projective variety $Y/k$ with associated Brauer class $\beta \in H^2(\mc Y, \mathbf{G}_m)$.  Let $K\in D(\cX \times \cY )^{(-1,1)}$ be a complex inducing an equivalence $D(X, \alpha )\rightarrow D(Y, \beta )$.

\begin{thm}\label{T:theorem3} The Fourier-Mukai equivalence $K$ induces an isomorphism $\cR ^0_{\cX }\rightarrow \cR ^0_{\cY }$.
\end{thm}

\begin{rem} The precise manner in which $K$ defines the isomorphism is explained in section \ref{S:section7}.  Intuitively, the map on $\cR ^0$ should be viewed as sending an autoequivalence $\rho $ of $D(X, \alpha )$ to the autoequivalence $\Phi ^K\circ \rho \circ (\Phi ^{K})^{-1}$ of $D(Y, \beta )$.
\end{rem}

In the last section we also discuss a description of gerbes and twisted derived categories over abelian varieties using the autormorphism groups of gerbes and descent, which we expect to use in future work on twisted derived categories of abelian varieties.

\subsection{Acknowledgements}
The author was partially supported by NSF FRG grant DMS-2151946 and a grant from the Simons Foundation.

\section{Twisted sheaves}\label{S:section2}

Before beginning the proofs of the above theorems, 
we review some facts about twisted sheaves and gerbes.  The results of this section are well-known to experts; in particular, a number of them can be extracted from \cite{MR2309155}.

\begin{lem}\label{L:1.1} Let $X$ be a scheme and let $\mc X_N$ be a $\mu _N$-gerbe for some $N>0$ with associated $\mathbf{G}_m$-gerbe $\mc X$.  Let $i\colon \cX _N\rightarrow \cX$ be the natural map.  Then the pullback functor $i^*\colon D(\cX )^{(1)}\rightarrow D(\cX _N)^{(1)}$ on categories of $1$-twisted sheaves is an equivalence with inverse the functor $i_*^{(1)}$ sending $\mc F\in D(\cX _N)^{(1)}$ to the $1$-twisted part of $i_*\mc F$.
\end{lem}
\begin{proof}
There are natural maps $\text{id}\rightarrow i_*^{(1)}i^*$ and $i^*i_*^{(1)}\rightarrow \text{id}$ and to verify that they are equivalences we may work fppf locally on $X$.  It therefore suffices to consider the case when $\cX _N = B\mu _{N, X}$ and $i$ is the natural map $B\mu _{N, X}\rightarrow B\mathbf{G}_{m, X}$.  In this case the result is immediate.
\end{proof}

\begin{lem}\label{L:1} Let $f\colon Y\rightarrow S$ be a proper flat morphism of algebraic spaces and let $\cY \rightarrow Y$ be a $\mathbf{G}_m$-gerbe, which is the pushout of a $\mu _N$-gerbe for some $N>0$.  Assume that the map $\mls O_S\rightarrow f_*\mls O_Y$ is an isomorphism, and that the same holds after arbitrary base change $S'\rightarrow S$.  Let $\cSec (\cY/Y)$ be the stack over $S$ which to any $T\rightarrow S$ associates the groupoid of sections $s\colon Y_T\rightarrow \cY _T$.  Then $\cSec (\cY /Y)$ is an algebraic stack which is a $\mathbf{G}_m$-gerbe over an algebraic space $\Sec (\cY /Y)$.
\end{lem}
\begin{proof} 
Let $\cY _N$ be a $\mu _N$-gerbe with pushout $\cY $.  The stack $\cSec (\cY /Y)$ is equivalent to the stack $\cPic ^{(1)}_{\cY }$ classifying $1$-twisted sheaves on $\cY$, and by \ref{L:1.1} this stack is in turn equivalent to the stack $\cPic ^{(1)}_{\cY _N}$ classifying $1$-twisted sheaves on $\cY _N$.  The result follows from these observations and \cite[1.1]{MR2480703}.
\end{proof}

\begin{rem}\label{R:2.3} Recall that if $X/k$ is a smooth projective variety over a field $k$ then every $\mathbf{G}_m$-gerbe over $X$ is torsion and therefore is the pushout of a $\mu _N$-gerbe for some $N>0$.
\end{rem}

\begin{pg}\label{P:2.4} Let $\cX \rightarrow X$ and $\cY \rightarrow Y$ be $\mathbf{G}_m$-gerbes over smooth projective $k$-schemes and let $K\in D(\cX \times \cY)^{(-1,1)}$ be a complex defining a functor
$$
\Phi ^K\colon D(\cX )^{(1)}\rightarrow D(\cY )^{(1)}.
$$
Let $A$ (resp. $B$)  denote $K^\vee \otimes (\omega _Y|_{\cY\times \cX})[\text{\rm dim}_Y]\in D(\cY\times \cX )^{(-1,1)}$ (resp. $K^\vee \otimes (\omega _X|_{\cY \times \cX })[\text{\rm dim}_X]\in D(\cY \times \cX )^{(-1,1)}$).  We will need the following mild generalization of classical results (e.g. \cite[1.2]{bondal1995semiorthogonaldecompositionalgebraicvarieties}).
\end{pg}

\begin{lem} The functor $\Phi ^{A}\colon D(\cY )^{(1)}\rightarrow D(\cX )^{(1)}$ (resp. $\Phi ^B$) is left (resp. right) adjoint to $\Phi ^K$. In particular, if $\Phi ^K$ is an equivalence then $\Phi ^A = \Phi ^B$ and this functor defines the inverse equivalence.   
\end{lem}
\begin{proof}
We prove that $\Phi ^A$ is left adjoint, leaving the very similar proof that $\Phi ^B$ is right adjoint to the reader.

  Let $\pi _2\colon \cX \times \cY\rightarrow \cX\times Y$ be the projection.  Let $F\in D(\cX )^{(1)}$ and $G\in D(\cY )^{(1)}$ be objects, and consider the diagram
  $$
  \xymatrix{
  & \cX \times \cY \ar[dd]_-{\pi _1}\ar[rd]^-{p_{\cY }}\ar[lddd]_-{p_{\cX }}& \\
  && \cY \ar[dd]^-{\pi _{\cY }} \\
  & \cX\times Y \ar[ld]^-{\bar p_{\cX }}\ar[rd]_-{\bar p_{Y }}& \\
  \cX&& Y .}
  $$
 Since $Y$ is smooth and projective, the functor $\rL \bar p_{\cX }^*(-)\lotimes \rL \bar p_Y^*\omega _Y[\text{dim}(Y)]\colon  D(\cX )^{(1)}\rightarrow D(\cX \times Y)^{(1)}$ is right adjoint to $\rR \bar p_{\cX *}\colon D(\cX \times Y)^{(1)}\rightarrow D(\cX )^{(1)}$.  Since the pullback functor $D(\cX \times Y)^{(1)}\rightarrow D(\cX \times \cY )^{(1, 0)}$ is fully faithful we conclude that 
 \begin{align*}
\text{Hom}_{\cX }(\Phi ^{A}(G), F) & \simeq \text{Hom}_{\cX }(\rR \bar p_{\cX *}(\rL p_{\cY }^*G\lotimes  A), F)\\
 & \simeq  \text{Hom}_{\cX \times \cY}(\rL p_{\cY}^*G\lotimes A,\rL p_{\cX }^*F\lotimes p_{\cY }^*\omega _Y[\text{dim}(Y)])\\
 & \simeq  \text{Hom}_{\cX \times \cY}(\rL p_{\cY }^*G, \rL p_{\cX}^*F\lotimes K)\\
 & \simeq  \text{Hom}_{\cY }(G, \Phi ^K(F)).
 \end{align*}
 This isomorphism is functorial in both $F$ and $G$, and therefore realizes $\Phi ^{A}$ as a left adjoint of $\Phi ^P$. 
\end{proof}

\begin{pg}\label{P:2.6}
The adjunction map $\Phi ^A\circ \Phi ^K\rightarrow \text{id}$ is realized as follows.  The composition $\Phi ^A\circ \Phi ^K$ is defined by the complex $\rR \text{pr}_{13*}(\rL \text{pr}_{12}^*K\lotimes \rL \text{pr}_{23}^*A)$, where the $\text{pr}_{ij}$ are the various  projections from $\mls X\times \mls Y\times \mls X$,  and the identity functor is given by $(\Delta _*\mls O_{\cX })^{(-1,1)}$.  The adjunction map is then given by the map
$$
\rL \Delta ^*\rR \text{pr}_{13*}(\rR \text{pr}_{12}^*K\lotimes \rL \text{pr}_{23}^*A)\simeq \rR \text{pr}_{1*}(K\lotimes K^\vee \lotimes \omega _Y[\text{dim}(Y)])\rightarrow \mls O_{\cX },
$$
induced by the evaluation map $K\otimes K^\vee \rightarrow \mls O_{\cX \times \cY  }$ and the trace map $\rR \text{pr}_{1*}(\omega _Y[\text{dim}(Y)])\rightarrow \mls O_{\cX }$. 
\end{pg}

\begin{lem}\label{L:2.6} Let $S$ be a scheme and let $K\in D(\cX _S\times _S\cY _S)^{(-1,1)}$ be a relatively perfect complex.  Then there exists an open subscheme $U\subset S$ such that a geometric point $\bar s\rightarrow S$ factors through $U$ if and only if $K_{\bar s}$ defines an equivalence $D(\cX _{\bar s})^{(1)}\rightarrow D(\cY _{\bar s})^{(1)}$.
\end{lem}
\begin{proof}
Consider the map
\begin{equation}\label{E:adjunctionmap}
\rR \text{pr}_{13*}(\rL \text{pr}_{12}^*K\otimes \rL \text{pr}_{23}^*A)\rightarrow (\Delta _*\mls O_{\cX })^{(-1,1)}
\end{equation}
realizing the adjunction map $\Phi ^A\circ \Phi ^K\rightarrow \text{id}$.  Then by the preceding discussion it suffices to show that there exists an open subset $U\subset S$ such that a geometric point $\bar s\rightarrow S$ factors through $U$ if and only if \eqref{E:adjunctionmap} induces an isomorphism in the fiber.  This follows from the derived version of Nakayama's lemma \cite[\href{https://stacks.math.columbia.edu/tag/0G1U}{Tag 0G1U}]{stacks-project}.
\end{proof}

\section{Proof of Theorem \ref{T:theorem1} (1) and (2)}

\begin{pg} We work in the setting of \ref{T:theorem1}.  So $X$ is a smooth proper geometrically connected scheme over a field $k$ and  $\pi \colon \cX \rightarrow X$ is a $\mathbf{G}_m$-gerbe.

Consider the $\mathbf{G}_m^2$-gerbe $\cX\times _X\cX$ over $X$.  There is a line bundle $\mc K$ on $\cX \times _X\cX$ defined by the $\mathbf{G}_m$-torsor which to a scheme $T$ and two maps $t_1, t_2\colon T\rightarrow \cX$ associates the $\mathbf{G}_m$-torsor $\underline {\text{Isom}}(t_1, t_2)$ over $T$.  Note that this sheaf $\mc K$ is $(-1,1)$-twisted and therefore defines a morphism $\cX \times _X\cX \rightarrow B\mathbf{G}_m$ compatible with the morphism $m\colon \mathbf{G}_m^2\rightarrow \mathbf{G}_m$ ($(u, v)\mapsto u^{-1}v$).
\end{pg}

\begin{lem}\label{L:0} The induced map
$$
\cX \times _X\cX \rightarrow \cX \times _XB\mathbf{G}_{m, X}
$$
is an isomorphism.
\end{lem}
\begin{proof}
    Indeed this is a morphism of gerbes over $X$ and the map $\mathbf{G}_m^2\rightarrow \mathbf{G}_m^2$ sending $(u,v)$ to $(u, u^{-1}v)$ is an isomorphism.
\end{proof}

\begin{pg} Let $\cAut _\cX$ denote the fibered category over the category of $k$-schemes which to any $T/k$ associates the groupoid of isomorphisms of $\cX _T \rightarrow \cX_T$ inducing the identity map $\mathbf{G}_m\rightarrow \mathbf{G}_m$ on stabilizer groups.  It is straightforward to verify that this is, in fact, a stack for the \'etale topology.

Note that for such an equivalence $\sigma \colon \cX _T\rightarrow \cX _T$ an automorphism of $\sigma $ is given by a lifting $\tilde \sigma \colon \cX \rightarrow \mc I_{\cX } = \cX \times \mathbf{G}_m$ of $\sigma $ to the inertia stack of $\cX $.  Since $X$ is smooth proper and geometrically connected we have $\mathbf{G}_m(\cX _T) = \mathbf{G}_m(T)$.  It follows that $\cAut _\cX $ is a $\mathbf{G}_m$-gerbe over a sheaf of groups $\Aut _\cX $.
\end{pg}

\subsection{Proof of \ref{T:theorem1} (1)}
If $\sigma \colon \cX _T\rightarrow \cX _T$ is an object of $\cAut _{\cX}(T)$ for a $k$-scheme $T$ then by passing to $\mathbf{G}_m$-rigidifications we get an induced automorphism $\bar \sigma \colon X_T\rightarrow X_T$.  This defines a morphism 
$$
c\colon \cAut _\cX \rightarrow \Aut _X,
$$
and therefore also a morphism $\bar c\colon \Aut _\cX \rightarrow \Aut _X$.

    Let $\alpha \colon X\times \Aut _X\rightarrow X\times \Aut _X$ be the universal automorphism over $\Aut _X$, and let $\cY \rightarrow X\times \Aut _X$ be the $\mathbf{G}_m$-gerbe given by the difference of $\cX_{\Aut _X}$ and $\alpha ^*\cX _{\Aut _X}$.  Then as a stack over $\Aut _X$ we have $\cAut (\cX)\simeq \cSec (\cY/X\times \Aut _X)$.  Now as noted in \ref{R:2.3} above,   the gerbe $\cX$ is the pushout of a $\mu _N$-gerbe for some $N>0$, and therefore the same is true for $\cY $.  We can therefore apply \ref{L:1} with $S  = \Aut _X$, $Y = X\times \Aut _X$, and the stack $\cY$, to conclude that $\cAut (\cX )$ is an algebraic stack locally of finite type over $k$.  This proves \ref{T:theorem1} (1). \qed

\subsection{Proof of \ref{T:theorem1} (2)}
The fibers of $\bar c$ can be understood as follows.  For a $k$-scheme $T$ and automorphism $\sigma \colon \cX _T\rightarrow \cX _T$ with induced automorphism $\bar \sigma \colon X_T\rightarrow X_T$ the groupoid of all automorphisms $\cX _T\rightarrow \cX _T$ over $\bar \sigma $ can be identified with the groupoid of liftings
$$
\xymatrix{
& \cX _T\times _{X_T}\cX _T\ar[d]_-{\text{pr}_1}\\
\cX _T\ar@{-->}[ru]\ar[r]^-\sigma & \cX _T,}
$$
which using \ref{L:0} identifies the groupoid of automorphisms over $\bar \sigma $ with the groupoid of $0$-twisted sheaves on $\cX _T$, or equivalently with $\cPic (X_T)$.    From this \ref{T:theorem1} (2) follows. \qed 

\begin{rem}
In fact the above discussion defines an action
$$
\cAut _{\cX }\times \cPic _{X/k}\rightarrow \cAut _{\cX}
$$
which upon passing to rigidifications determine an action
$$
\Aut _{\cX }\times \Pic _{X/k}\rightarrow \Aut _{\cX }.
$$
for which the induced map
$$
\Aut _{\cX }\times \Pic _{X/k}\rightarrow \Aut _{\cX }\times _{\Aut _X}\Aut _{\cX }
$$
is an isomorphism.

Note that taking $\sigma = \text{id}$ in the above we get a morphism of stacks $\cPic _{X/k}\hookrightarrow \cAut _{\cX }$ which induces a homomorphism $\Pic _{X/k}\hookrightarrow \Aut _{\cX }$ defining the above action. 
\end{rem}

\section{Proof of Theorem \ref{T:theorem1} (3)}

It suffices to consider the case when $k$ is algebraically closed.
Let $\cX_N$ be a $\mu _N$-gerbe inducing $\cX$, with $N$ invertible in $k$, and let $\bar \alpha \in \Aut ^0_{X}(k)$ be an automorphism of $X$ in the connected component of the identity.  To prove that $\bar \alpha $ lifts to $\Aut ^0_{\cX }(k)$ it suffices to show that the two gerbes $\bar \alpha ^*\cX _N$ and $\cX _N$ are isomorphic.

For this we consider the universal case.  Let $\bar a:X\times \Aut ^0_{X}\rightarrow X\times \Aut ^0_X$ be the universal automorphism, and let $\rho :X\times \Aut ^0_X\rightarrow X$ be the composition of $\bar a$ with the first projection.  Let $[\cX _N]\in H^2(X, \mu _N)$ be the class of $\cX _N$ and let $\rho ^*[\cX _N]\in H^0(\Aut ^0_X, \rR ^2\pi _*\mu _N)$ be the pullback, where $\pi :X\times \Aut ^0_X\rightarrow \Aut ^0_X$ is the second projection.  The fiber of $\rho ^*[\cX _N]$ at a point $\bar \alpha \in \Aut ^0_X(k)$ is the class $[\bar \alpha ^*\cX _N]\in H^2(X, \mu _N)$.  If $q:X\times \Aut ^0_X\rightarrow X$ is the first projection, then we can also consider the constant class $q^*[\cX _N]\in H^0(\Aut ^0_X, \rR ^2\pi _*\mu _N)$, and it suffices to show that the two classes $q^*[\cX _N]$ and $\rho ^*[\cX _N]$ are equal.  To see this note that since $X/k$ is smooth and proper and $N$ is invertible in $k$ the sheaf $\rR ^2\pi _*\mu _N$ is a locally constant \'etale sheaf on $\Aut ^0_X$, and therefore it suffices to show that the two classes are equal at a single point of $\Aut ^0_X$.  Since they agree at the identity the result follows. \qed

\section{Proof of theorem \ref{T:theorem2}}

Fix a field $k$.

\begin{pg} Let $\cX $ be a $\mathbf{G}_m$-gerbe over a smooth proper $k$-scheme $X$ and let $\cX ^{(2)}$ denote the pushout of the $\mathbf{G}_m^2$-gerbe $\cX \times \cX $ over $X\times X$ along the homomorphism $\mathbf{G}_m^2\rightarrow \mathbf{G}_m$ sending $(u,v)$ to $u^{-1}v$.  So $\cX ^{(2)}$ is a $\mathbf{G}_m$-gerbe over $X\times X$.

If $\alpha \colon \cX \rightarrow \cX$ is an automorphism with associated graph $\Gamma _\alpha := (\text{id}, \alpha )\colon \cX \rightarrow \cX \times \cX$ then the composition of $\Gamma _\alpha $ with  the projection $q\colon \cX \times \cX \rightarrow \cX ^{(2)}$ induces the trivial homomorphism $\mathbf{G}_m\rightarrow \mathbf{G}_m$, and therefore descends to a morphism $\gamma _\alpha  \colon X\rightarrow \cX ^{(2)}$ such that the square
$$
\xymatrix{
\cX \ar[r]^-{\Gamma _\alpha }\ar[d]& \cX \times \cX \ar[d]^-q\\
X\ar[r]^-{\gamma _\alpha }& \cX ^{(2)}}
$$
is cartesian.  In particular, taking $\alpha $ the identity map we get a morphism $\delta \colon X\rightarrow \cX ^{(2)}$ over the diagonal map $\Delta \colon X\rightarrow X\times X$.

For a $k$-scheme $S$ pullback along the base change $q_S:(\cX \times \cX)_S \rightarrow \cX ^{(2)}_S$ of $q$ to $S$ induces an equivalence of categories $D(\cX _S^{(2)})^{(n)}\simeq D((\cX \times \cX )_S)^{(-n,n)}$ for all $n\in \mathbf{Z}$.  In particular, for $n=1$ we get an equivalence $D(\cX _S^{(2)})^{(1)}\simeq D((\cX \times \cX )_S)^{(-1,1)}$ sending $(\gamma _{\alpha *}\mls O_X)^{(1)}$ to $(\Gamma _{\alpha *}\mls O_{\cX })^{(-1,1)}$.
\end{pg}

\begin{defn} A complex $K\in D(\cX ^{(2)})^{(1)}$ \emph{satisfies the Rouquier condition} if it is of the form $(\gamma _{\alpha *}\mls O_X)^{(1)}$ for some $\alpha \in \cAut _{\cX }(k)$.
\end{defn}

\begin{rem} Note that this is equivalent to the condition that the pullback $q^*K\in D(\cX \times \cX )^{(-1,1)}$ is an object of $\mls R_{\cX }$ (defined in \ref{P:Rdef}).    Thus $\cR _{\cX }$ can be viewed as the stack which to any $k$-scheme $S$ associates the groupoid of perfect complexes $K\in D(\cX _S^{(2)})^{(1)}$ such that for all geometric points $\bar s\rightarrow S$ the fiber $K_{\bar s}\in D(\cX _{\bar s}^{(2)})^{(1)}$ satisfies the Rouquier condition.
\end{rem}

\begin{pg} For $K\in D(\cX ^{(2)})^{(1)}$ define
$$
\Phi ^K\colon D(\cX )^{(1)}\rightarrow D(\cX )^{(1)}
$$
to be the functor sending $F\in D(\cX )^{(1)}$ to $\rR \text{pr}_{2*}(\rL \text{pr}_1^*F\otimes q^*K).$
Note here that since $\rL \text{pr}_1^*F$ is $(1,0)$-twisted the complex $\rL \text{pr}_1^*F\otimes q^*K$ is $(0,1)$-twisted on $\cX \times \cX $.
\end{pg}

\begin{prop}\label{P:keyprop} Let $S$ be a scheme and $K\in D(\cX ^{(2)})^{(1)}$ a complex.   There exists a unique open subset $U\subset S$ such that a geometric point $\bar s\rightarrow S$ factors through $U$ if and only if the complex $K_{\bar s}\in D(\cX ^{(2)}_{\bar s})^{(1)}$ satisfies the Rouquier condition.  Furthermore, the restriction $K_U\in D(\cX _U^{(2)})$ is of the form $(\gamma _{\alpha *}\mls O_{X})^{(1)}$ for a unique automorphism $\alpha :\cX _U\rightarrow \cX _U$.
\end{prop}
\begin{proof}
By \ref{L:2.6} it suffices to consider the case when $\Phi ^K$ is an equivalence.

Let $\bar s\rightarrow S$ be a geometric point such that the complex $K_{\bar s}$ satisfies the Rouquier condition.  We show that there exists an \'etale neighborhood $W\rightarrow S$ of $\bar s$, an automorphism $\alpha \colon \cX _W\rightarrow \cX _W$ defining a point in $\cAut _{\cX }(W)$ such that $K|_{\cX _W^{(2)}} = (\gamma _{\alpha *}\mls O_{\cX })^{(1)}$. This suffices for proving the proposition. 

Fix an integer $N$ such that there exists a $\mu _N$-gerbe $\mls X_N$ inducing $\mls X$, and let $K_N$ denote the corresponding complex on $\mls X_N$. 
 By the derived Nakayama lemma \cite[\href{https://stacks.math.columbia.edu/tag/0G1U}{Tag 0G1U}]{stacks-project} there exists an open neighborhood around  the image of $\bar s$ over which the complex $K$ is a sheaf flat over $S$ concentrated in degree $0$.   Replacing $S$ by this open neighborhood we may assume that $K$ is a sheaf.  By a standard limit argument we may further replace $S$ by the strict henselization of $S$ at $\bar s$, and then using the Grothendieck existence theorem for $1$-twisted sheaves on $\mls X$, which holds by the corresponding result for $\mls X_N$, it suffices to prove the following deformation theoretic result.  Let $A'\rightarrow A$ be a surjective morphism of artinian local rings over $S$ with kernel $J$ annihilated by $\mathfrak{m}_A$ and residue a field $k$ (note that then $J$ can be viewed as a $k$-vector space).  Suppose given an automorphism $\alpha _A\colon \cX _{A}\rightarrow \cX _A$ over $A$ such that $K|_{\cX _A^{(2)}}\simeq (\gamma _{\alpha _A*}\mls O_{X_A})^{(1)}$.  We then show that there exists a lifting $\alpha _{A'}\colon \cX _{A'}\rightarrow \cX _{A'}$ of $\alpha _A$ such that $K|_{\cX _{A'}^{(2)}}\simeq  (\gamma _{\alpha _{A'}*}\mls O_{X_{A'}})^{(1)}$.

For this note first that $K|_{\cX _{A'}^{(2)}}$ is a sheaf concentrated in degree $0$ and flat over $A'$:  This follows from noting that we have a distinguished triangle
$$
(\gamma _{\alpha _k*}\mls O_{X_k})^{(1)}\otimes _kJ\rightarrow K|_{\cX _{A'}^{(2)}}\rightarrow (\gamma _{\alpha _A*}\mls O_{X_A})^{(1)}\rightarrow (\gamma _{\alpha _k*}\mls O_{X_k})^{(1)}\otimes _kJ[1]
$$
and looking at the associated long exact sequence of cohomology sheaves.  Furthermore, applying $R\mls Hom(K, -)$ to this sequence and observing that $\mls Hom _{\cX _{A'}^{(2)}}(K|_{\cX _{A'}^{(2)}}, (\gamma _{\alpha *}\mls O_{X_k})^{(1)})$ is the $0$-twisted sheaf given by $\mls O_{\Gamma _{\bar \alpha _A}}$ (the pullback to $\cX _{A}^{(2)}$ of the structure sheaf of the graph of $\bar \alpha _A\colon X_A\rightarrow X_A$) we get an  exact sequence
$$
\xymatrix{
0\ar[r]& \mls O_{\Gamma _{\bar \alpha _k}}\otimes J\ar[r] & \mls Hom _{\cX _{A'}^{(2)}}(K|_{\cX _{A'}^{(2)}}, K|_{\cX _{A'}^{(2)}})\ar[r]& \mls O_{\Gamma _{\bar \alpha _A}}\ar[r]& 0,}
$$
where the right exactness follows from the observation that the surjection $\mls O_{\cX _{A'}^{(2)}}\rightarrow \mls O_{\Gamma _{\bar \alpha _A}}$ factors through  the natural map
\begin{equation}\label{E:5.6.1}
\mls O_{\cX _{A'}^{(2)}}\rightarrow \mls Hom _{\cX _{A'}^{(2)}}(K|_{\cX _{A'}^{(2)}}, K|_{\cX _{A'}^{(2)}}).
\end{equation}
From this it also follows that the map \eqref{E:5.6.1} is surjective and that the target is a coherent $0$-twisted sheaf of algebras defining a closed subscheme $Z\subset X_{A'}^2$ flat over $A'$ whose reduction to $A$ is the graph of an automorphism.  We conclude that $Z$ is the graph of an automorphism $\bar \alpha _{A'}$ of $X_{A'}$ lifting $\bar \alpha _A$.  Consider the $\mathbf{G}_m$-gerbe $\mls G\rightarrow X_{A'}$ given by the fiber product of the diagram
$$
\xymatrix{
& \cX _{A'}^{(2)}\ar[d]\\
X_{A'}\ar[r]^-{(\text{id}, \bar \alpha _{A'})}& X_{A'}\times _{A'}X_{A'}.}
$$
It follows from the preceding discussion that $K|_{\cX _{A'}}$ is a $1$-twisted sheaf on $\mls G$ and therefore defines a section $s\colon X\rightarrow \mls G$.  Composing with the map $\mls G\rightarrow \cX ^{(2)}_{A'}$ and making the base change $\cX _{A'}\times _{A'}\cX _{A'}\rightarrow \cX _{A'}^{(2)}$ we get a morphism $\cX _{A'}\rightarrow \cX _{A'}\times _{A'}\cX _{A'}$ whose projection to the first factor is the identity and whose projection to the second factor is an automorphism $\alpha _{A'}\colon \cX _{A'}\rightarrow \cX _{A'}$ lifting $\alpha _A$ and such that $K|_{\cX _{A'}}\simeq (\gamma _{\alpha _{A'}*}\mls O_{ X_{A'}})^{(1)}$.  Furthermore, the construction shows that $\alpha _{A'}$ is unique.
\end{proof}

\subsection{Proof of \ref{T:theorem2}}
Note that statement (1) in \ref{T:theorem2} follows from statement (2) and \ref{T:theorem1}.
In light of \ref{P:keyprop} to prove statement (2) it suffices 
 to show that if $\alpha :\cX _S\rightarrow \cX _S$ is an automorphism of $\cX _S$ for a $k$-scheme $S$ with associated complex $(\gamma _{\alpha *}\mls O_{X_S})^{(1)}\in D(\cX _S^{(2)})(1)$ then we recover $\alpha $ uniquely from $(\gamma _{\alpha *}\mls O_{X_S})^{(1)}$.  Let $\bar \alpha :X_S\rightarrow X_S$ be the automorphism defined by $\alpha $ and let $\gamma _{\bar \alpha }:X_S\rightarrow (X\times X)_S$ be its graph.   The fiber product of the diagram
$$
\xymatrix{
& \cX ^{(2)}_S\ar[d]\\
X_S\ar[r]^-{\gamma _{\bar \alpha }}& (X\times X)_S}
$$
is canonically isomorphic to the $\mathbf{G}_m$-gerbe $\cX ^{-1}\wedge \bar \alpha ^*\cX $.  We therefore get a commutative diagram
$$
\xymatrix{
X_S\ar@{=}[rd]\ar[r]^-{t_\alpha }& \cX ^{-1}_S\wedge \bar \alpha ^*\cX_S \ar[d]\ar[r]& \cX_S^{(2)}\ar[d]\\
& X_S\ar[r]^-{\gamma _{\bar \alpha }}& (X\times X)_S,}
$$
where the square is cartesian and $t_\alpha $ is the trivialization defined by $\alpha $.  From this we see that  $(\gamma _{\alpha *}\mls O_{X_S})^{(1)}$ is the pushforward of a unique $1$-twisted invertible sheaf on $\cX _S^{-1}\wedge \bar \alpha ^*\cX _S$, and this $1$-twisted sheaf defines the map $t_{\alpha }$ and therefore also the map $\alpha $. This discussion also implies that the scheme-theoretic support of $(\gamma _{\alpha *}\mls O_{X_S})^{(1)}$ is a closed substack of $\cX _S^{(2)}$ which is a $\mathbf{G}_m$-gerbe over the graph of $\bar \alpha $, and $(\gamma _{\alpha *}\mls O_{X_S})^{(1)}$ defines a $1$-twisted invertible sheaf on this gerbe defining $\alpha $.  Therefore we can recover $\alpha $ uniquely from $(\gamma _{\alpha *}\mls O_{X_S})^{(1)}$. This shows that for any scheme $S$ the functor $\cAut _{\cX }(S)\rightarrow \cR _{\cX }(S)$ sending $\alpha $ to $(\gamma _{\alpha *}\mls O_{X_S})^{(1)}$ is an equivalence of categories, proving \ref{T:theorem2}. \qed

For later use we also record the following observation:
\begin{lem} Let $K_1, K_2\in \mls R_{\cX }(T)$ be two objects over a scheme $T$.  Then the complex
$$
K_1*K_2:= \rR \text{pr}_{13*}(\rL p_{12}^*K_1\lotimes \rL p_{23}^*K_2)\in D((\cX \times \cX )_T)^{(-1,1)}
$$
is an object of $\mls R_{\cX }(T)$.
\end{lem}
\begin{proof}
It suffices to show this in the case when $T$ is the spectrum of an algebraically closed field, where the result is immediate.
\end{proof}

\begin{rem} The product $K_1*K_2$ is called the \emph{convolution product} of $K_1$ and $K_2$.  It corresponds to composition of autoequivalences $D(X, \alpha )\rightarrow D(X, \alpha )$.  Furthermore, the map $\cAut _{\cX }\rightarrow \cR _{\cX }$ takes composition of autoequivalences to the convolution product of kernels, as follows immediately from the definition.
\end{rem}

\section{The action of $\cR _{\cX }$ on $D(\cX )$}

\begin{prop}\label{P:4.1} Let $\mls M$ be an invertible sheaf on $X$ with associated automorphism $\alpha _{\mls M}\colon \cX \rightarrow \cX$.  Then the induced functor $\alpha _{\mls M}^*\colon D(\cX)\rightarrow D(\cX)$ is given by the functors $\otimes \mls M^{\otimes i}\colon D(\cX )^{(i)}\rightarrow D(\cX )^{(i)}$.
\end{prop}
\begin{proof}
Fix a covering $X = \cup _iU_i$ over which we have trivializations $\sigma _i\colon \mls M|_{U_i}\rightarrow \mls O_{U_i}$ so the line bundle $\mls M$ is described by units $u_{ij}\in \Gamma (U_{ij}, \mls O_{U_{ij}}^*)$, where $U_{ij}:= U_i\cap U_j$, satisfying the cocycle condition on triple overlaps.  Let $x_i\colon \cX _{U_i}\rightarrow \cX$ be the projection.  The map $\alpha _{\mls M}\colon \cX \rightarrow \cX $ is  described by descent as follows.  The trivialization of $\mls M$ over $U_i$ identifies the composition $\alpha _{\mls M}\circ x_i$ with $x_i$ so the additional data specifying $\alpha _{\mls M}$ are isomorphisms
$$
\sigma _{\mls M}^{ij}\colon x_i|_{\cX _{U_{ij}}}\rightarrow x_j|_{\cX _{U_{ij}}}
$$
satisfying the cocycle condition on triple overlaps.  If $\sigma ^{ij}$ denotes the tautological isomorphism then it follows from the construction of $\alpha _{\mls M}$ that $\sigma _{\mls M}^{ij}$ is given by $u_{ij}\cdot \sigma ^{ij}$ (viewing $u_{ij}$ as an automorphism of $x_j|_{\cX _{U_{ij}}}$).

From this it follows that if $\mc F$ is a coherent sheaf on $\cX$ then $\alpha _{\mls m}^*\mc F$ is the coherent sheaf on $\cX$ whose pullback to $\cX _{U_i}$ is the restriction $\mc F_i$ of $\mc F$ to $\cX _{U_i}$ but whose descent data is given by composing the tautological desecnt data $\lambda _{ij}\colon \mc F_i|_{\cX _{U_{ij}}}\rightarrow \mc F_j|_{\cX _{U_{ij}}}$ with the automorphism of $\mc F_j|_{\cX _{U_{ij}}}$ given by the automorphism $u_{ij}$ of $x_j|_{\cX _{U_{ij}}}$.  In particular, if $\mc F$ is $r$-twisted for some $r$ then the descent data is given by $u_{ij}^r\cdot \lambda _{ij}$, which is exactly the descent data for $\mc F\otimes \mls M^{\otimes r}$.  From this the result follows.
\end{proof}

\begin{cor}\label{C:4.2}  Let $\Delta _{\mls M}\colon \cX \rightarrow \cX \times \cX $ be the graph $\text{\rm id}\times \alpha _{\mls M}$ of $\alpha _{\mls M}$.  Then  $\Delta _{\mls M*}\mls O_{\cX } = \oplus _{n\in \mathbf{Z}}((\Delta _*\mls O_{\cX })^{(-n, n)}\otimes \text{\rm pr}_2^*\mls M^{\otimes -n}).$
\end{cor}
\begin{proof}
This follows from \ref{P:4.1} and noting that we have a commutative square 
$$
\xymatrix{
\cX \ar[r]^-{\Delta _{\mls M}}\ar[d]_-{\text{id}}& \cX \times \cX \ar[d]^-{\text{id}\times \alpha _{\mls M^{-1}}}\\
\cX \ar[r]^-{\Delta }& \cX \times \cX ,}
$$
with the vertical maps isomorphisms.
\end{proof}

\begin{cor} Let $\mc M$ be a line bundle on $\cX$ and let $\Delta _{\mc M}\colon \cX \rightarrow \cX \times \cX $ be the graph of the induced automorphism.  Then the $(-1,1)$-twisted sheaf $(\Delta _{\mc M*}\mls O_{\cX })^{(-1,1)}\in D(\cX \times \cX )^{(-1,1)}$ is the pullback under $q^*$ of $(\delta _*\mc M^{-1})^{(1)}\in D(\cX ^{(2)})^{(1)}$.
\end{cor}
\begin{proof}
This follows from \ref{C:4.2}.
\end{proof}

\subsection{The action on $D(X, \alpha )$.}

The group stack $\cAut _{\cX }$ acts on the triangulated category $D(X, \alpha )$.  It is unclear to us if this action is faithful.  However, we have the following result:

\begin{prop}\label{P:faithful} Assume that $k$ is algebraically closed, and let $r\geq 1$ be an integer such that $\cX $ admits a $1$-twisted vector bundle $\mls E$ of rank $r$. Let $\sigma \in \cAut _{\cX }(k)$ be an automorphism of the stack inducing the identity functor on $D(X, \alpha )$.  Then $\sigma $ maps to $\text{\rm Pic}_X[r]\subset \mathbf{R}_{\cX }$.
\end{prop}
\begin{proof}
Note first that the automorphism $\bar \sigma :X\rightarrow X$ induced by $\sigma $ must be the identity. Indeed for a closed point $x\in X(k)$ let $\mls E_x$ be the $1$-twisted sheaf on $\cX$ given by tensoring $\mls E$ with the skyscraper sheaf at $x$.  Then $\sigma *\mls E_x\simeq \mls E_x$ (since $\sigma $ acts as the identity on $D(X, \alpha )$) and looking at supports we get that $\bar \sigma (x) = x$.  

It follows that $\sigma $ is given by tensoring with a line bundle $\mls M$ on $X$.  This line bundle furthermore has the property that for any other line bundle $\mls L$ on $X$ we have
$$
(\mls E\otimes \mls L)\otimes \mls M\simeq \mls E\otimes \mls L,
$$
since $\mls E\otimes \mls L\in D(X, \alpha )$ and $\sigma $ acts trivially on this category.  Taking determinants we find that $\mls L^{\otimes r}\otimes \mls M^{\otimes r}\simeq \mls L^{\otimes r}$, from which we conclude that $\mls M^{\otimes r}\simeq \mls O_X$.
\end{proof}

\begin{example}
    If one allows $X$ to be a Deligne-Mumford stack, and not just a scheme, then one can make an example of a gerbe $\mls X$ for which the action of $\cAut _{\mls X}$ on $D(X, \alpha )$ is not faithful as follows.
    Let $N>0$ be an integer and let $k$ be an algebraically closed field in which $N$ is invertible.  Let $A$ denote the group $\mu _N\times \mathbf{Z}/(N)$.  We view $\mathbf{Z}/(N)$ as the Cartier dual of $\mu _N$ and for $\zeta \in \mu _N$ and $a\in \mathbf{Z}/(N)$ we write $a(\zeta )$ for the value of $a$ on $\zeta $.  Consider the ``Heisenberg group'' $\mls G$, which as a scheme is  $\mathbf{G}_m\times A$ but with product given by 
    $$
    (u, (\zeta , a))*(u', (\zeta ', a')) = (uu'a(\zeta '), \zeta \zeta ', a+a').
    $$
The group scheme $\mls G$ is a central extension
$$
1\rightarrow \mathbf{G}_m\rightarrow \mls G\rightarrow A\rightarrow 1.
$$
Let $X$ denote $BA$ and let $\mls X\rightarrow X$ be the $\mathbf{G}_m$-gerbe given by $B\mls G\rightarrow BA$.  By standard representation theory (see for example \cite[Proposition 3]{MR204427}) there exists a unique irreducible representation $V$ of $\mls G$ on which $\mathbf{G}_m$ acts by scalar multiplication and any representation of $\mls G$ on which $\mathbf{G}_m$ acts by scalars is a direct sum of copies of $V$.  Furthermore, the endomorphism ring of $V$ (as a representation) is $k$.  If $\mls V$ denotes the corresponding sheaf on $B\mls G$ then it follows that the functor
$$
\text{Hom}(\mls V, -):(\text{$1$-twisted quasi-coherent sheaves on $B\mls G$})\rightarrow \text{Vec}_k
$$
is an equivalence of categories.  Deriving this equivalence we see that
$$
\text{RHom}(\mls V, -):D(BA, \alpha )\rightarrow D(\text{Vec}_k)
$$
is an equivalence of triangulated categories, where $\alpha \in H^2(BA, \mathbf{G}_m)$ is the class of $B\mls G$.

Let $\mls L$ be a line bundle on $BA$ with associated $1$-dimensional representation $L$ of $A$, which we view also as a representation of $\mls G$, and let $\sigma :D(BA, \alpha )\rightarrow D(BA, \alpha )$ be the autoequivalence given by tensoring with $\mls L$.  To show that $\sigma $ is isomorphic to the identity functor it suffices to show that the two functors
$$
\text{RHom}(\mls V, -), \ \ \text{RHom}(\mls V, -)\circ \sigma 
$$
are isomorphic.  Since $\text{RHom}(\mls V, -)\circ \sigma \simeq \text{RHom}(\mls V\otimes \mls L^{-1} , -)$, for this it suffices in turn to show that the two representations $V$ and $V\otimes L$ are isomorphic.  This follows from noting that they are both representations of $\mls G$ of the same rank on which $\mathbf{G}_m$ acts by multiplication by scalars.
\end{example}

\section{Proof of theorem \ref{T:theorem3}}\label{S:section7}

We use an argument we learned from Christian Schnell in the untwisted case.

\begin{pg}\label{P:7.1}
As in \ref{P:2.4} let $A\in D(\cY \times \cX )^{(-1,1)}$ denote the complex $K^\vee \otimes (\omega _Y|_{\cY\times \cX})[\text{\rm dim}_Y]$. Consider the complex $\rL p_{12}^*A\lotimes \rL p_{34}^*K\in D(\cY \times \cX \times \cX \times \cY ) ^{(-1,1,-1,1)}$.  Applying the isomorphism 
$$
\cY \times \cX \times \cX \times \cY \rightarrow \cX \times \cX \times \cY \times \cY , \ \ (y_1, x_1, x_2, y_2)\mapsto (x_1, x_2, y_1, y_2)
$$
and using the equivalence $D(\cX ^{(2)}\times \cY ^{(2)})^{(-1,1)}\simeq D(\cX \times \cX \times \cY \times \cY )^{(1,-1,-1,1)}$ we get an object $\Omega \in D(\cX ^{(2)}\times \cY ^{(2)})^{(-1,1)}$, which defines a functor
$$
\Psi :D(\cX ^{(2)})^{(1)}\rightarrow D(\cY ^{(2)})^{(1)}.
$$
\end{pg}
\begin{rem} If we view objects of $D(\cX ^{(2)})^{(1)}$ (resp. $D(\cY ^{(2)})^{(1)}$)  as defining endofunctors of $D(X, \alpha )$ (resp. $D(Y, \beta )$) then $\Psi $ sends an endofunctor $E$ to $\phi ^K\circ E\circ (\Phi ^{K})^{-1}$, as follows from the description of the kernel of a composition of Fourier-Mukai functors.
\end{rem}

\begin{pg}
To prove \ref{T:theorem3} we show that for a $k$-scheme $S$ and object $\Sigma \in \cR ^0_{\cX }(S)\subset D(\cX ^{(2)}_S)^{(1)}$ the object $\Psi _S(\Sigma )\in D(\cY ^{(2)})^{(1)}$ is in $\cR ^0_{\cY }$, and furthermore that the induced functor
$$
\Psi ^{\cR ^0}:\cR ^0_{\cX }\rightarrow \cR ^0_{\cY }
$$
is compatible with convolution.  
\end{pg}

\subsection{Compatibility with convolution.}

For this it is useful to generalize the operation of convolution slightly.  For three $\mathbf{G}_m$-gerbes $\cX _i$ over smooth projective varieties $X_i$ ($i=1,2,3$) and objects $A\in D(\cX _1\times \cX _2)^{(-1,1)}$ and $B\in D(\cX _2\times \cX _3)^{(-1,1)}$ let $A*B\in D(\cX _1\times \cX _3)^{(-1,1)}$ denote the complex
$$
B*A:= \rR \text{pr}_{13*}(\rL \text{pr}_{12}^*A\lotimes \rL \text{pr}_{23}^*B).
$$

\begin{lem}\label{L:7.5} For $A\in D(\cX _1\times \cX _2)^{(-1,1)}$, $B\in D(\cX _2\times \cX _3)^{(-1,1)}$, and $C\in D(\cX _3\times \cX _4)^{(-1,1)}$ we have
$$
C*(B*A) \simeq (C*B)*A
$$
in $D(\cX _1\times \cX _4)^{(-1,1)}$.
\end{lem}
\begin{proof}
    Indeed expanding out the definition of the convolution product one finds that both sides are calculated by
    $$
    \rR \text{pr}_{14*}(\rL \text{pr}_{12}^*A\lotimes \rL \text{pr}_{23}^*B \lotimes \rL \text{pr}_{34}^*C),
    $$
    where the pushforward is from $\cX _1\times \cX _2\times \cX _3\times \cX _4$.
\end{proof}

Returning to the notation of \ref{P:7.1}.
\begin{lem}\label{L:7.6}
\begin{enumerate}
    \item  $A*K\simeq (\Delta _{\cX *}\mls O_{\cX })^{(-1,1)}$ and $K*A\simeq (\Delta _{\cY*}\mls O_{\cY})^{(-1,1)}$.
    \item  For any $U\in D(\cX \times \cX )^{(-1,1)}$ we have $U*(\Delta _{\cX *}\mls O_{\cX })^{(-1,1)}\simeq (\Delta _{\cX *}\mls O_{\cX })^{(-1,1)}*U\simeq U$.
\end{enumerate}
\end{lem}
\begin{proof}
Statement (1) is a reformulation of the discussion in \ref{P:2.6}, and statement (2) is immediate from the definitions.
\end{proof}

With this notation and identifying $D(\cX ^{(2)})^{(1)}$ with $D(\cX \times \cX )^{(-1,1)}$ and similarly for $\cY $,  for $U\in D(\cX \times \cX )^{(-1,1)}$ we have $\Psi (U)$ equal to $K*U*A$, where the associativity of convolution \ref{L:7.5} is reflected in the notation.  For $U, V\in D(\cX \times \cX )^{(-1,1)}$ we then have
$$
\Psi (V)*\Psi (U) \simeq K*V*A*K*U*A\simeq K*V*U*A\simeq \Psi (V*U),
$$
where the middle isomorphism is by \ref{L:7.6}.  This is the sought-after compatibility with convolution.

\subsection{Completion of proof of \ref{T:theorem3}}

By \ref{L:7.6} we have $\Psi ((\Delta _{\cX *}\mls O_{\cX })^{(-1,1)})\simeq (\Delta _{\cY *}\mls O_{\cY })^{(-1,1)}$.   From this and \ref{P:keyprop} we find that there exists a maximal nonempty, and therefore dense, open subset $U\subset \mathbf{R}_{\cX }^0$ such that for every geometric point $\bar u\rightarrow U$ with corresponding object $E_{\bar u}\in D(\cX \times \cX )^{(-1,1)}$ we have $\Psi (E_{\bar u})\in \cR _{\cY }^0$.  Furthermore, the open subset $U$ is closed under convolution.  Since the map 
$$
U\times U\rightarrow \mathbf{R}_{\cX }^0, \ \ ([E], [F])\mapsto [E*F]
$$
is surjective (this is a general fact about connected group schemes) we conclude that $U = \mathbf{R}_{\cX }^0$ which implies \ref{T:theorem3}. \qed 

\section{Twisted derived category of abelian variety}\label{S:section8}

In the case of an abelian variety the collection of gerbes as well as twisted sheaves can be described quite concretely, as we explain in this section.

Let $k$ be an algebraically closed field and let $A/k$ be an abelian variety.

\begin{lem} Let $\mls X\rightarrow A$ be a $\mathbf{G}_m$-gerbe of order $e$. Then the pullback of $\mls X$ along the multiplication map $[e]:A\rightarrow A$ is trivial. 
 In particular, for any $\mathbf{G}_m$-gerbe $\mls X\rightarrow A$ there exists an isogeny $\tau :A'\rightarrow A$ such that $\tau ^*\mls X$ is a trivial gerbe over $A'$.
\end{lem}
\begin{proof}
It suffices to show that multiplication $[n]:A\rightarrow A$ on the abelian variety (where $n>0$ is a positive integer) induces multiplication by $n^2$ on $\text{Br}(A)$. 
For this it suffices, in turn, to consider the case when $n = \ell $ is a prime number.
If $\ell $ is invertible in the ground field  then this follows from the isomorphism $\bigwedge ^2H^1(A, \mathbf{F}_\ell )\simeq H^2(A, \mathbf{F}_\ell )$ (\cite[15.1]{MR861974} and the universal coefficient theorem) and the fact that $[\ell ]^*$ equals multiplication by $\ell $ on $H^1(A, \mathbf{F}_\ell )$.

For $\ell = p$ the argument, which we learned from a MathOverflow post\footnote{\url{https://mathoverflow.net/questions/363979/involution-action-on-brauer-group-of-an-abelian-variety}}, is more complicated.  Consider the ``Hoobler sequence'' \cite{Hoobler}
$$
\xymatrix{
0\ar[r]&\mathbf{G}_m/\mathbf{G}_m^p\ar[r]^-{\dlog }& Z^1_A\ar[r]^-{\text{id}-C}& \Omega ^1_A\ar[r]& 0,}
$$
where $Z^1_A$ is the sheaf of closed $1$-forms and $C$ is the Cartier operator.  This is a sheaf on the \'etale site of $A$. 
 Combining this with the Kummer sequence, which is a sheaf on the fppf site of $A$, 
$$
0\rightarrow \mu _p\rightarrow \mathbf{G}_m\rightarrow \mathbf{G}_m\rightarrow 0
$$
we find that if $\epsilon :A_{\text{fppf}}\rightarrow A_\et $ is the projection then
$$
R\epsilon _*\mu _p[1]\simeq (\xymatrix{Z^1_A\ar[r]^-{\text{id}-C}& \Omega ^1_A}).
$$
In particular, we get a natural map
$$
H^2(A, \mu _p)\rightarrow H^1(A, Z^1_A),
$$
and therefore also a map $H^2(A, \mu _p)\rightarrow H^2_{\text{dR}}(A)$ which is injective by \cite[1.2]{MR563467}. This inclusion is functorial in $A$, and now since $\bigwedge ^2H^1_{\text{dR}}(A)\simeq H^2_{\text{dR}}(A)$ we again conclude the result.
\end{proof}

\subsection{$\mathbf{G}_m$-gerbes via descent}

\begin{pg}
    Fix an isogeny $\tau :A'\rightarrow A$.  Let $\BR (A)$ denote the $2$-category of $\mathbf{G}_m$-gerbes over $A$, so that the set of isomorphisms classes of $\BR (A)$ is the Brauer group of $A$, and let $\BR (A/A')\subset \BR(A)$ be the sub-$2$-category of gerbes which are trivial (but not trivialized) over $A'$. We can describe the category $\BR (A/A')$ more explicitly via descent as follows.
\end{pg}

\begin{pg}
Define a $2$-category $\mc C$ as follows.

{\bf Objects.} Collections of data $(\gamma, g)$, where $\gamma :\Sigma \rightarrow \mls Pic ^0_{A'}$ is a functor and $g$ is an isomorphism as follows.   Let 
$\mls S$ be the line bundle on $A'_\Sigma $ obtained by pullback along $\gamma :\Sigma \rightarrow \mls Pic ^0_{A'}$ from the universal line bundle on $A'\times \mls Pic ^0_{A'}$.  Consider the maps
$$
p_i:A'_{\Sigma ^2}\rightarrow A'_\Sigma , \ \ (a', \sigma _1, \sigma _2)\mapsto (a', \sigma _i), \ \ i=1,2,
$$
$$
m:A'_{\Sigma ^2}\rightarrow A'_{\Sigma }, \ \ (a', \sigma _1, \sigma _2)\mapsto (a', \sigma _1+\sigma _2),
$$
$$
t_1:A'_{\Sigma ^2}\rightarrow A'_{\Sigma ^2}, \ \ (a', \sigma _1, \sigma _2)\mapsto (a'+\sigma _1, \sigma _1, \sigma _2).
$$
 Then $g$ is an isomorphism
$$
g:(t_1^*p_2^*\mls S)\otimes p_1^*\mls S\rightarrow m^*\mls S.
$$
This isomorphism is further required to satisfy the following cocycle condition.
Define maps
$$
m_{12}:A'_{\Sigma ^3}\rightarrow A'_{\Sigma ^2}, \ \ (a', \sigma ^{\prime \prime }, \sigma ', \sigma )\mapsto (a', \sigma ^{\prime \prime }+\sigma ', \sigma ),
$$
$$
m_{23}:A'_{\Sigma ^3}\rightarrow A'_{\Sigma ^2}, \ \ (a', \sigma ^{\prime \prime }, \sigma ', \sigma )\mapsto (a', \sigma ^{\prime \prime },\sigma '+ \sigma ),
$$
$$
\tilde t_1:A'_{\Sigma ^3}\rightarrow A'_{\Sigma ^3}, \ \ (a', \sigma ^{\prime \prime }, \sigma ', \sigma )\mapsto (a'+\sigma ^{\prime \prime }, \sigma ^{\prime \prime }, \sigma ', \sigma ),
$$
 for $1\leq i<j\leq 3$ let
$$
p_{ij}:A'_{\Sigma ^3}\rightarrow A'_{\Sigma ^2}
$$
be the map induced by the projection map $\Sigma ^3\rightarrow \Sigma ^2$ onto the $i$-th and $j$-th factors, and for $1\leq i\leq 3$ let
$$
\tilde p_i:A'_{\Sigma ^3}\rightarrow A'_\Sigma 
$$
be given by the $i$-th projection $\Sigma ^3\rightarrow \Sigma $.
The cocycle condition is then that the following diagram commutes:
\begin{equation}\label{E:cocycle}
\xymatrix{
\tilde t_1^*(p_{23}^*(t_1^*p_2^*\mls S)\otimes p_1^*\mls S)\otimes \tilde p_1^*\mls S\ar[r]^-{\simeq }\ar[d]^-{\tilde t_1^*p_{23}^*(g)}&m_{12}^*(t_1^*p_2^*\mls S)\otimes p_{12}^*(t_1^*p_2^*\mls S)\otimes p_{12}^*(p_1^*\mls S)\ar[d]^-{1\otimes p_{12}^*(g)}\\
\tilde t_1^*(p_{23}^*m^*\mls S)\otimes \tilde p_1^*\mls S\ar[d]^-{\simeq }& m_{12}^*(t_1^*p_2^*\mls S)\otimes p_{12}^*(m^*\mls S)\ar[d]^-{\simeq }\\
m_{23}^*(t_1^*p_2^*\mls S\otimes p_1^*\mls S)\ar[d]^-{m_{23}^*(g)}& m_{12}^*(t_1^*p_2^*\mls S\otimes p_1^*\mls S)\ar[d]^-{m_{12}^*(g)}\\
m_{23}^*m^*\mls S\ar[r]^-{\simeq }& m_{12}^*m^*\mls S.}
\end{equation}

Two such pairs  $(\gamma, g)$ and $(\gamma ', g')$ are defined to be equivalent, denoted $(\gamma, g)\sim (\gamma ', g')$, if there exists an isomorphism $u:\mls S\rightarrow \mls S'$ between the associated line bundles on $A'_\Sigma $ such that the diagram
$$
\xymatrix{
(t_1^*p_2^*\mls S)\otimes p_1^*\mls S\ar[d]^-{t_1^*p_2^*u\otimes p_1^*u}\ar[r]^-g& m^*\mls S\ar[d]^-{m^*u}\\
(t_1^*p_2^*\mls S')\otimes p_1^*\mls S'\ar[r]^-{g'}& m^*\mls S'}
$$
commutes.  Note that such an isomorphism $u$, which is equivalent to the data of an isomorphism of functors $\gamma \rightarrow \gamma '$, is unique if it exists.

{\bf Morphisms.}
For a line bundle $\mls M$ let $\mls U^{\mls M}$ denote the invertible sheaf on $A'_\Sigma $ given by
$$
\mls U^{\mls M}:= \rho ^*\mls M\otimes p_1^*\mls M^{-1},
$$
where $\rho :A'_\Sigma \rightarrow A'$ is the action map.  Then there is a canonical isomorphism
$$
v_{\mls M}:(t_1^*p_2^*\mls U^{\mls M})\otimes p_1^*\mls U^{\mls M}\rightarrow m^*\mls U^{\mls M},
$$
over $\Sigma ^2$.  On scheme-valued points $(\sigma ', \sigma )\in \Sigma ^2$ this is given by the isomorphism
$$
t_{\sigma '}^*(t_{\sigma }^*\mls M\otimes \mls M^{-1})\otimes t_{\sigma '}^*\mls M\otimes \mls M^{-1}\simeq t_{\sigma +\sigma '}^*\mls M\otimes \mls M^{-1}
$$
arising from the fact that $t_{\sigma '}^*t_{\sigma }^*(-)\simeq t_{\sigma +\sigma '}^*(-)$.  It follows from the definition that the analogue of the diagram \eqref{E:cocycle} for $(\mls U^{\mls M}, v_{\mls M})$ commutes.  Given an object $(\gamma, g)$ let $(\gamma ^{\mls M}, g^{\mls M})$ be the object with associated line bundle $\mls S\otimes \mls U^{\mls M}$ and $g^{\mls M}$ given by $g\otimes v_{\mls M}$.  We define the category of morphisms 
$$
\text{HOM}_{\mls C}((\gamma, g), (\gamma ', g'))
$$
to be the groupoid of line bundles $\mls M$ on $A'$ for which $(\gamma ^{\mls M}, g^{\mls M})\sim (g', \gamma ')$.

Composition is given by tensor product of line bundles.
\end{pg}

\begin{rem} The data of the pair $(\gamma, g)$ can also be described as follows.  The line bundle $\mls S$ is equivalent to specifying for any $k$-scheme $T$ and point $\sigma \in \Sigma (T)$ a line bundle $\mls S_\sigma $ on $A'_T$ functorially in $T$.  With this notation the data of $g$ amounts to an isomorphism
$$
g_{\sigma ', \sigma }:t_{\sigma '}^*\mls S_\sigma \otimes \mls S_{\sigma '}\simeq \mls S_{\sigma '+\sigma }
$$
for any two points $\sigma ', \sigma \in \Sigma (T)$.  The cocycle condition \eqref{E:cocycle} then amounts to the statement that for any three points $\sigma ^{\prime \prime }, \sigma ', \sigma \in \Sigma (T)$ the diagram
$$
\xymatrix{
t_{\sigma ^{\prime \prime *}}(t_{\sigma '}^*\mls S_\sigma \otimes \mls S_{\sigma '})\otimes \mls S_{\sigma ^{\prime \prime }}\ar[r]^-{t_{\sigma ^{\prime \prime }}^*g_{\sigma ', \sigma }}\ar[d]_-{\simeq }& t_{\sigma ^{\prime \prime }}^*\mls S_{\sigma '+\sigma}\otimes \mls S_{\sigma ^{\prime \prime }}\ar[rd]^-{g_{\sigma ^{\prime \prime }, \sigma '+\sigma }}\\
t_{\sigma ^{\prime \prime }+\sigma '}^*\mls S_\sigma \otimes t_{\sigma ^{\prime \prime }}^*\mls S_\sigma '\otimes \mls S_{\sigma ^{\prime \prime }}\ar[r]^-{g_{\sigma ^{\prime \prime }, \sigma '}}& t_{\sigma ^{\prime \prime }+\sigma '}^*\mls S_\sigma \otimes \mls S_{\sigma ^{\prime \prime }+\sigma '}\ar[r]^-{g_{\sigma ^{\prime \prime }+\sigma ', \sigma }}& \mls S_{\sigma ^{\prime \prime }+\sigma '+\sigma }}
$$
commutes, and the same holds after arbitrary base change $T'\rightarrow T$.
\end{rem}

\begin{pg} There is a functor
\begin{equation}\label{E:Gfunctor}
G:\mls C\rightarrow \BR (A/A')
\end{equation}
defined as follows.

Fix $(\gamma, g)$ with associated invertible sheaf $\mls S$.  Define a $\mathbf{G}_m$-gerbe $\mls X$ as follows.  For a $k$-scheme $T$ let $\mls X(T)$ be the groupoid of data $(P, f, \mls R, b)$, where $P\rightarrow T$ is a $\Sigma $-torsor, $f:P\rightarrow A'$ is a $\Sigma $-equivariant morphism, $\mls R$ is a line bundle on $P$, and 
$$
b:f_\Sigma ^*\mls S\simeq \rho ^*\mls R\otimes p^*\mls R^{-1}
$$
is an isomophism of line bundles ove $P_\Sigma $, where 
$\rho :P_\Sigma \rightarrow P$ (resp. $p:P_\Sigma \rightarrow P$) is the action map (resp. projection).  The isomorphism $b$ is futhermore required to be compatible with $g$ in the sense that for any two points $\sigma , \sigma '\in \Sigma (T)$ the diagram
$$
\xymatrix{
f^*(\sigma ^{\prime *}\mls S_\sigma \otimes \mls S_{\sigma '})\ar[r]^-{\sigma ^{\prime *}b_{\sigma }\otimes b_{\sigma '}}\ar[d]_-{f^*g_{\sigma ', \sigma }}& \sigma ^{\prime *}(\sigma ^*\mls R\otimes \mls R^{-1})\otimes \sigma ^{\prime *}\mls \otimes \mls R^{-1}\ar[d]^-{\simeq }\\
f^*\mls S_{\sigma '+\sigma }\ar[r]^-{b_{\sigma '+\sigma }}& (\sigma '+\sigma )^*\mls R\otimes \mls R^{-1}}
$$
commutes, and the sense holds after base change $T'\rightarrow T$ and points over $T'$.

A morphism
$$
(P, f, \mls R, b)\rightarrow (P',f', \mls R', b')
$$
is an isomorphism of $\Sigma $-torsors $\lambda :P\rightarrow P'$ such that $f'\circ \lambda = f$ and an isomorphism of line bundles $\lambda ^b:\lambda ^*\mls R'\simeq \mls R$ compatible with the isomorphisms $b$ and $b'$.

Note that since the $\Sigma $-action on $A'$ is faithful the maps $f$ and $f'$ are necessarily monomorphisms and therefore $\lambda $ is unique if it exists.  Furthermore, the isomorphisms $\lambda ^b$ is unique up to multiplication by an element  $u\in H^0(P, \mls O_P^*)$ satisfying $\rho ^*u=p^*u$ (because of the compatibility with $b$) which is equivalent to saying that $u$ is $\Sigma $-invariant and therefore an element of $H^0(T, \mls O_T^*)$.

This shows that $\mls X(T)$ is a groupoid and that the automorphisms of any object are canonically identified with $\mathbf{G}_m(T)$.   There is a natural map $\mls X\rightarrow A$ sending $(P, f, \mls R, b)$ over a scheme $T$ to the $T$-point $T = P/\Sigma \rightarrow A'/\Sigma = A$ defined by $f$. To verify that $\mls X$ is a gerbe over $A$ observe that an object $(P, f, \mls R, b)\in \mls X(T)$ can \'etale locally on $T$ be described as follows.  Replacing $T$ by an \'etale cover we may assume that $P$ is trivial.  Fixing such a trivialization we get an isomorphism $P\simeq \Sigma _T$ and $f$ corresponds simply to a point $a'\in A'(T)$.  Let $\mls R_0$ be the line bundle obtained by restricting $\mls R$ to the zero section of $\Sigma _T$.  Then we see that for any section $\sigma \in \Sigma (T)$ the map $b$ defines an isomorphism 
$$
\mls R_0\otimes a^{\prime *}\mls S_\sigma \simeq \mls R_\sigma,
$$
where $\mls R_\sigma $ is the fiber of $\mls R$ over $\sigma \in \Sigma \simeq P$.
It follows that $\mls R$ and $b$ are determined by this formula and $\mls R_0$, and the data $(P, f, \mls R, b)$ is determined simply by the $\Sigma $-orbit of $a'$; that is, the induced point of $A(T)$.
\end{pg}

\begin{lem} The gerbe $\mls X\times _AA'$ is trivial and therefore $\mls X\in \BR(A/A').$
\end{lem}
\begin{proof}
    Indeed observe that for an object $(P, f, \mls R, b)\in \mls X(T)$ over a scheme $T$ with associated point $a\in A(T)$, the torsor $P$ can be recovered as the fiber product $A'\times _{A, a}T$.  So $\mls X\times _AA'$ can be viewed as the stack which to any scheme $T$ associates the groupoid of data $(P, f, \mls R, b)$ together with a trivialization of $P$.  As we saw in the previous discussion such data is equivalent to a line bundle $\mls R_0$ on $T$, which defines an isomorphism $\mls X\times _AA'\simeq B\mathbf{G}_{m, A'}$.
\end{proof}

\begin{prop} The functor \eqref{E:Gfunctor} is an equivalence of $2$-categories.
\end{prop}
\begin{proof}
Let us first show that every object is in the essential image.  Let $\mls X\in \BR (A/A')$ be an object, and fix a trivialization $\tau :\mls X\times _AA'\simeq B\mathbf{G}_{m, A'}.$  The action of $\Sigma $ on $A'$ over $A$ defines an action of $\Sigma $ on $\mls X\times _AA'$ via the second factor, and therefore using $\tau $ we get a map
$$
\Sigma \rightarrow \cAut _{B\mathbf{G}_{m, A'}}\simeq A'\times \mls Pic^0_{A'}.
$$
Define $\gamma _{(\mls X, \tau )}$ to be the second factor of this map.    If we write this map on scheme-valued points as 
$$
\sigma \mapsto (\sigma  , \mls S_\sigma )
$$
then the compatibility with composition is given by functorial isomorphisms
$$
g_{\sigma ', \sigma }:t_{\sigma '}^*\mls S_\sigma \otimes \mls S_{\sigma '}\simeq \mls S_{\sigma '+\sigma },
$$
satisfying the cocycle condition.  
That is, we get an object $(\gamma, g)\in \mls C$.  Furthermore, it is straightfoward to verify that the associated gerbe of this data is isomorphic to $\mls X$.  In fact, this construction shows that $(\gamma , g)$ is uniquely associated to the data $(\mls X, \tau )$. Two different choices of $\tau $ differ by a line bundle on $\mls M$, and it follows from the construction that this action of $\mls Pic_{A'}$ is compatible with the action on $\mls C$, which implies the full faithfulness as well.
\end{proof}

\subsection{A description of $D(A, \alpha )$}

Continuing with the preceding notation, fix a gerbe $\mls X\in \BR(A/A')$ and a trivialization $\tau :\mls X_{A'}\simeq B\mathbf{G}_{m, A'}$ defining a pair $(\gamma , g)$.

The associated line bundle $\mls S$ on $A'_\Sigma $ can also be characterized as follows.
For a $k$-scheme $T$  and $\sigma \in \Sigma (T)$ the following diagram is $2$-commutative
$$
\xymatrix{
B\mathbf{G}_{m, A'_T}\ar[r]^-{\tau }\ar[d]_-{(\sigma, \mls S_\sigma )}& \mls X\times _A{A'_T}\ar[d]^-{\text{id}\times \sigma }\\
B\mathbf{G}_{m, A'_T}\ar[r]^-{\tau }& \mls X\times _A{A'_T}.}
$$
Using \ref{P:4.1} we find that for a $1$-twisted quasi-coherent sheaf $\mls F$ on $\mls X_T$ we have a specified isomorphism
$$
v_\sigma :t_\sigma ^*\tau ^*\mls F\simeq \tau ^*\mls F\otimes \pi ^*\mls S_\sigma .
$$
Furthermore, if $\sigma '\in \Sigma (T)$ is a second point then the diagram
\begin{equation}\label{E:composition}
\xymatrix{
t_{\sigma '}^*t_\sigma ^*\tau  ^*\mls F\ar[r]^-{t_{\sigma '}^*v_\sigma }\ar[d]^-{\simeq}& t_{\sigma '}^*\tau  ^*\mls F\otimes \pi ^*t_{\sigma '}^*\mc L_\sigma \ar[r]^-{v_{\sigma '}}& \tau  ^*\mls F\otimes \pi ^*(t_{\sigma '}^*\mc L_\sigma \otimes \mc L_{\sigma '})\ar[d]^-{\gamma _{\sigma ', \sigma}}\\
t_{\sigma +\sigma '}^*\tau  ^*\mls F\ar[rr]^-{v_{\sigma +\sigma '}}&& \tau  ^*\mls F\otimes  \mc L_{\sigma +\sigma '}}
\end{equation}
commutes.

\begin{pg}
    The universal case ($T = \Sigma $ and $\sigma $ the identity map $\Sigma \rightarrow \Sigma $) of the preceding discussion gives the following. Let 
    $$
    T :B\mathbf{G}_{m, A'_\Sigma }\rightarrow B\mathbf{G}_{m, A'_\Sigma }
    $$
    be the universal translation map and let $\mls S$ be the line bundle on $A'_\Sigma $ obtained by pullback along $\gamma :\Sigma \rightarrow \mls Pic ^0_{A'}$ from the universal line bundle on $A'\times \mls Pic ^0_{A'}$.  We then have an isomorphism
    $$
    T^*(\tau  ^*\mls F)\simeq \tau  ^*\mls F\otimes \pi ^*\mls S
    $$
    over $B\mathbf{G}_{m, A'_\Sigma },$ and this map satisfies the compatibility with composition over $\Sigma ^2$ given by the diagram \eqref{E:composition}.  Twisting $\tau  ^*\mls F$ by the inverse of the standard character we get a quasi-coherent sheaf $\mls G$ on $A'_\Sigma $ with an isomorphism 
    $$
    V:t^*\mls G\simeq \mls G\otimes \mls S
    $$
    on $A'_\Sigma $, again satisfying the cocycle condition over $A'_{\Sigma ^2}$, where $t:A'\times \Sigma \rightarrow A'$ is the action map.
\end{pg}

\begin{pg} Let $\mls U$ denote the category of pairs $(\mls G, V)$, where $\mls G$ is a  quasi-coherent sheaf on $A'$ and
$$
V:t^*(\mls G|_{A'_\Sigma })\rightarrow \mls G|_{A'_\Sigma }\otimes \mls S
$$
is an isomorphism over $A'_\Sigma $ such that the cocycle condition holds over $\Sigma ^2$.  Morphisms $(\mls G, V)\rightarrow (\mls G', V')$ are morphisms of quasi-coherent sheaves $\mls G\rightarrow \mls G'$ respecting the isomorphisms $V$ and $V'$.

The preceding discussion defines a functor
\begin{equation}\label{E:descentfunctor}
\text{Qcoh}(\mls X)^{(1)}\rightarrow \mls U.
\end{equation}
\end{pg}

\begin{prop} The functor \eqref{E:descentfunctor} is an equivalence of categories.
\end{prop}
\begin{proof}
Since $A$ is the quotient of $A'$ by $\Sigma $ it follows that $\mls X$ is the quotient of $B\mathbf{G}_{m, A'}$, which is schematic over $\mls X$, by the action of $\Sigma $. By construction, tensoring with the standard character identifies the category $\mls U$ with the category of $\Sigma $-linearized $1$-twisted sheaves on $B\mathbf{G}_{m, A'}$.  From this and descent theory, the result follows.
\end{proof}

\begin{cor} The bounded derived category $D^b(\mls U)$ is equivalent to $D(A, \alpha )$.
\end{cor}
\qed

\bibliographystyle{amsplain}
\bibliography{bibliography}{}

\end{document}